\documentclass{article}[10pt]
\usepackage{theorem}
\usepackage{eurosym}
\usepackage{amssymb}
\usepackage{float}
\usepackage{amsmath}
\usepackage{amsfonts}
\usepackage[american]{babel}
\usepackage{verbatim}
\usepackage{geometry}

\usepackage[colorlinks]{hyperref}
\usepackage{subfigure}
\usepackage{mathtools}
\usepackage{srcltx}
\usepackage{authblk}
\usepackage{csquotes}

\newcommand{\R}        {\mathbb {R}}

\newcommand{\eps}      {\epsilon}
\newcommand{\lap}      {\bigtriangleup}

\newcommand{\grad}     {\nabla}
\newcommand{\noi}      {\noindent}
\newcommand{\del}      {\partial}

\newcommand{\Prod}{\mathop{\prod}\limits}

\newtheorem{theorem}{Theorem}

{\theorembodyfont{\rmfamily}

}

\newtheorem{lemma}[theorem]{Lemma}

{\theorembodyfont{\rmfamily}
\newtheorem{remark}[theorem]{Remark}
}

\newenvironment{proof}[1][Proof]{\noindent\textbf{#1} }{\ \rule{0.5em}{0.5em}}

\newcommand*\lam{\lambda}

\newcommand{\rlemma}[1]{Lemma~\ref{#1}}
\newcommand{\rth}[1]{Theorem~\ref{#1}}

\def\supp{\operatorname{supp}}
\def\div{\operatorname{div}}

\begin{document}

\title{A mixed eigenvalue problem on domains tending to infinity in several directions}
\date{}
\author[1]{Prosenjit Roy\thanks{prosenjit@iitk.ac.in}}
\author[2]{Itai Shafrir\thanks{shafrir@technion.ac.il}}
\affil[1]{Department of Mathematics and Statistics,
  Indian Institute of Technology, Kanpur 208016, INDIA}
\affil[2]{Department of Mathematics, Technion - I.I.T., 32000 Haifa, ISRAEL}

%
%
%
%
%
%
%
%

\maketitle

\smallskip

\begin{abstract}
The aim of this article is to  analyze the asymptotic behaviour of the
eigenvalues of elliptic operators in divergence form with mixed
boundary type conditions for  domains that  become unbounded in
several directions, while they stay  bounded  in some directions
(cylindrical domains).   The limiting behavior of  such eigenvalues is
shown to depend on an ensemble of  eigenvalue problems defined on a domain that is
unbounded only in one direction. The asymptotic behavior of the eigenfunctions are also discussed. This work is a continuation  of the work done in \cite{pr}.
\end{abstract}

\section{Introduction}\label{section:introduction}
Let $m\ge 2, \  p  \geq 1 $ and $\omega_1 \subset \R^m$, $\omega_2
\subset \R^p$ two bounded domains
 with $C^1$ boundary, such that $\omega_1$ contains the origin. By a domain we mean
 a nonempty open connected set. 
 For $\ell>0$ consider the cylindrical domain $\Omega_\ell =  \ell\omega_1 \times \omega_2 \subset \R^{m+p}$.
 A generic point $x\in\Omega_\ell $ is denoted  by $x=(X, \xi)$ where
 $X =(x_1,x')=  (x_1 , x_2, \ldots , x_m ) \in \R^m$ and $\xi =
 (\xi_1,  \  \xi_2, \ldots  ,\xi_p  ) \in \R^p$. We assume that  the
 $(m+p)\times(m+p)$ matrix
 
 \[
A = A(\xi) =
\begin{pmatrix}
A_{11}(\xi) & A_{12}(\xi) \\
A_{12}^T(\xi) & A_{22}(\xi)
\end{pmatrix},
\qquad
\begin{array}{l}
A_{11}(\xi)\in\mathbb{R}^{m\times m},\;
A_{12}(\xi)\in\mathbb{R}^{m\times p},\\[2pt]
A_{22}(\xi)\in\mathbb{R}^{p\times p},
\end{array}
\]
whose elements are measurable and bounded functions on $\omega_2$, is  symmetric and uniformly elliptic, that is, there exists a constant  $c_A> 0$ such that
\begin{equation}
\label{ellip}
(A(\xi) y)\cdot y \geq c_A \|y\|^2,
\text{ for all }   y\in \R^{m+p}\,,\text{ a.e. }\xi\in \omega_2.
\end{equation}
 We also assume that the matrix norms of  the family of matrices $\{A(\xi)\}_{\xi\in\omega_2}$ are~ uniformly bounded, that is
\begin{equation}
\label{mod}
\|A(\xi)\|= \sup_{ x\in \R^{m+p}\setminus \{0\}}\frac{\|A(\xi)x\|}{\|x
  \|} \leq C_A, \text{ a.e. } \xi\in \omega_2.
\end{equation}

Let us write $\partial\Omega_\ell=\Gamma_\ell\cup \gamma_\ell$ where
 \begin{equation}\label{eq:decom}
\Gamma_\ell = \del( \ell\omega_1)\times{}\omega_2  \text{ and }   \gamma_\ell = \ell\omega_1\times{}\partial\omega_2.
\end{equation}
 Denote by $\lambda_\ell^k$   the $k$-th eigenvalue for the mixed
 Neumann-Dirichlet problem
\begin{equation}
\label{eq:5}
\left\{
\begin{aligned}
  - &\textrm{div}(A \nabla u)=\sigma u &\text{ in }\Omega_\ell,\\
     &u=0 &\text{ on }\gamma_\ell,\\
     & (A \nabla u)\cdot\tilde{\nu}=0 &\text{ on }\Gamma_\ell.
\end{aligned}
\right.
\end{equation}
Here  $\tilde{\nu}$ denotes the outward unit normal to $\Gamma_\ell$.
When $k = 1$ we will denote for short 
$\lambda_\ell=\lambda_\ell^1$  and let  $u_\ell$ denote the (unique)
positive  normalized eigenfunction, i.e., $u_\ell>0$ and satisfies $\int_{\Omega_\ell} u_\ell^2
= 1$.  We have the following variational characterization of 
$\lambda_\ell$:
\begin{equation}
  \label{eq:19}
 \lambda_\ell= \inf_{  \left\{ 0\ne u \in H^1(\Omega_\ell)\, | \, u = 0 \text{ on
     }\gamma_\ell,  (A \nabla u)\cdot\tilde{\nu}=0 \text{ on }\Gamma_\ell\right\} }  \frac{\int_{\Omega_\ell}(A \grad u)\cdot \grad u}{\int_{\Omega_\ell}u^2}\,.
\end{equation}

Our goal is to understand the asymptotic behavior of $\lambda_\ell$
when the parameter $\ell$  tends to infinity, which means  that the
domains $\Omega_\ell$ tend to become unbounded  in the  first $m$
directions. For that matter the following eigenvalue problem turns out
to be relevant (as in \cite{pr}). Let $\mu_1$ denote the first
eigenvalue and  let $W$ be the (unique)
positive normalized eigenfunction for the operator
$-\div(A_{22}\nabla u)$, with Dirichlet boundary conditions, on the cross section $\omega_2$ of
$\Omega_\ell$, i.e.,
\begin{equation}
\label{eq:54}
\left\{
\begin{aligned}
  -&\textrm{div}(A_{22} \nabla W)=\mu_1 W \quad\text{ in }\omega_2,\\
     &W=0 \text{ on }\del\omega_2,\\
     & \int_{\omega_2}W^2 =1.
\end{aligned}
\right.
\end{equation}
Here $\div=\div_\xi$ denotes
the divergence operator in the $\xi$ variable. As in \cite{pr}, the problem \eqref{eq:54} will play an important role
in the study of problem \eqref{eq:5}.

In \cite{pr} the authors studied the  problem \eqref{eq:5} for the
case where $m=1$ and $k=1$ in complete generality. It was found that  the
limiting behaviour  of $\lambda_\ell$ is determined by  a minimization
problem set on a semi-infinite cylinder  and in particular   the
presence of Neumann  boundary conditions gives rise to a ``gap
phenomenon" ($\lim_{\ell \rightarrow \infty} \lambda_\ell < \mu_1 $ in
the second part of \rth{crs} below) in the limiting behaviour of $\lambda_\ell$. More precisely  their result reads as follow:

\begin{theorem}\emph{[Chipot-Roy-Shafrir]}
  \label{crs}When $m=1$ we have
  $$\lim_{\ell \rightarrow \infty}
  \lambda_\ell = \min \left\{Z^{+\infty}, Z^{-\infty} \right\},$$
  where
  \begin{equation}
    \label{eq:3}
   Z^{\pm \infty}= \inf_{\{0\ne u\in H^1({\R}_{\pm} \times\omega_2)  |
     u= 0 \text{ on } {\R}_{\pm}\times\del \omega_2 \} }
   \frac{\int_{{\R}_{\pm}\times\omega_2 }(A\grad u)\cdot \grad u}{\int_{{\R}_{\pm}\times\omega_2}u^2}\,.
  \end{equation}
  Furthermore, if $A_{12}\grad_\xi W$ does not equal to $0$ a.e.\,in $\omega_2$ then $\lim_{\ell
    \rightarrow \infty} \lambda_\ell < \mu_1$, while if $A_{12}\grad_\xi W=0$ a.e.\,in $\omega_2$ then $\lambda_\ell =
  \mu_1$ for all $\ell$.
\end{theorem}
 This result  stands in sharp contrast to the problem with full
 Dirichlet boundary conditions, for which  it was shown in \cite{dir}  that all
 the eigenvalues converge to $\mu_1$ when $\ell$ goes to infinity. We state their result as
 we will use it later on.
\begin{theorem}\emph{[Chipot-Rougirel]}
\label{chipot}
One has for some constant $C>0$,  $$\mu_1 \leq \sigma_\ell^k \leq
\mu_1 + \frac{C}{\ell^2},$$ where $\sigma_\ell^k$ is the $k-th$
eigenvalue  of operator $-\div(A\grad u)$ on $\Omega_\ell$ with Dirichlet boundary conditions.
\end{theorem}
\begin{remark}
  \label{rem:RQ}
  In the special case $k=1$  we have the well known variational
  characterization
  \begin{equation}
    \label{eq:26}
    \sigma_\ell^1=\inf_{u\in
      H^1_0(\Omega_\ell)}\frac{\int_{\Omega_\ell}(A\nabla u)\cdot
      \nabla u}{\int_{\Omega_\ell}u^2}\,.
  \end{equation}
\end{remark}
 The main aim of this article  is  to study  the limiting behaviour of
 $\lambda_\ell^k$ for the case $m >1$, or in other words, when the 
 cylinder  becomes unbounded in more than one direction. In
 particular, we will show that the limiting behaviour in this case is
 determined by an appropriate minimization  problem set on domain of the
 type  $(-\infty, 0)\times \omega_2$. Note that the
 domain
 $(-\infty, 0)\times \omega_2$ has  
 $m-1$ dimensions less than the domain $\Omega_\ell$ on which the original problem
 \eqref{eq:5} is defined. This is clearly very helpful from the point of view of numerical
 analysis.

 \par
 To state our main theorem we will need to introduce a family of
 eigenvalue problems, of the same type as the problem in
 \eqref{eq:3}, each of them is associated with some direction
 $\nu\in S^{m-1}$. So for each $\nu\in S^{m-1}$ let
 \begin{equation}\label{xc}
   Z^\nu = \inf_{  \left\{ 0\ne u \in H^1((-\infty, 0)\times  \omega_2)  \ | \ u = 0 \  \textrm{on} \ (-\infty, 0)\times  \del\omega_2 \right\} }  \frac{\int_{(-\infty, 0)\times  \omega_2}(A_\nu \grad u)\cdot \grad u}{\int_{(-\infty, 0)\times  \omega_2}u^2}\,,
 \end{equation}
 where $A _\nu=A_\nu(\xi)$ is the $(p+1)\times(p+1)$ matrix given by
 \begin{equation}
   \label{eq:4}
  A _\nu=\begin{pmatrix} (A_{11}\nu)\cdot \nu   &   \nu^T A_{12}\\
(\nu^T A_{12})^T & A_{22}
\end{pmatrix}
\,.
 \end{equation}
It is easy to deduce from our assumption \eqref{ellip} that all the
matrices $\{A_\nu(\xi)\}_{\nu,\xi}$ are uniformly elliptic, and the
same $c_A$ as in \eqref{ellip} can be taken as
a common lower bound for the least eigenvalue.

 Our main result is then
\begin{theorem}\emph{[Main Result]}
\label{main}  We have
\begin{equation}\label{ty}
\lim_{\ell \rightarrow \infty} \lambda_\ell = \inf_{\nu \in S^{m-1}} Z^\nu\,.
\end{equation}
Furthermore, if \begin{equation}\label{co} A_{12} \grad_{\xi} W  \not\equiv 0  \ a.e. \ \textrm{in}  \ \omega_2 \end{equation}then $\lim_{\ell \rightarrow \infty} \lambda_\ell <  \mu_1$ or else $\lambda_\ell = \mu_1,$ for all $\ell$.
\end{theorem}
Notice   that  \rth{main} is a natural generalization of \rth{crs} since for $m=1$ the relevant unit sphere is $S^0=\{-1,+1\}$.
 In \cite{pr}, under a certain symmetry assumption on the matrix $A$, it was
 also proved  that $\lambda_\ell^2$ has the same limit as
 $\lambda_\ell$, but no information was given on the third and higher
 eigenvalues. The result of \cite{pr}, was generalized recently in
 \cite{rrr} to $\lambda_\ell^k$ for all $k$  (under the same symmetry
 assumptions as in \cite{pr}). It turns out that for $m\ge2$ the problem for the higher
 eigenvalues is
 simpler and we are able to prove in this case 
 that $\lambda_\ell^k$ converges  to the same limit as
 $\lambda_\ell$, for all $k$. This immediately  leads to the existence
 of  a ``gap phenomenon" for all $k$ when \eqref{co} holds.
 \begin{theorem}\emph{[The Case of Higher Order Eigenvalues]}
   \label{high}
  We have
\begin{equation}\label{tyk}
\lim_{\ell \rightarrow \infty} \lambda_\ell^k = \inf_{\nu \in S^{m-1}}
Z^\nu\,,\text{ for all }k>1.
\end{equation}
In particular, when \eqref{co} holds, the value in \eqref{tyk} is
strictly lower than $\mu_1$.
 \end{theorem}
\par  The paper is  arranged as follows. In the next section we
 introduce various notation, auxiliary spaces and some known facts
 that  will be used throughout this paper.  In
 Section \ref{sec2}  we  provide  the  upper bound estimate for
 $\lambda_\ell$ by proving Theorem
 \ref{upperbound}. The  lower bound estimate   for $\lambda_\ell$ is
the objective of Section \ref{sec44}. Its proof in \rth{lb} provides
the necessary ingredient to conclude the proof of \rth{main}. In the
 last section, Section\,\ref{sec:high}, we 
 will  give the proof of Theorem \ref{high}, concerning the higher
 order eigenvalues.

 Various other problems of such  type  ($\ell \rightarrow \infty$) had
 been studied in the past. Beside the work mentioned above, it
 includes variational problem, second order elliptic equations, Stokes
 equation, problems involving fractional Laplacian, variational
 inequalities and many others. To obtain more idea about  of this type
 of works we refer to \cite{ b, delpino, mard, alex, indra,
   so,allaire, gt,rooy,  Karen} and the references mentioned there
 in. In most cases, the limiting behaviour of the parameter involved
 in such problems finds its connection with an appropriate problem set
 on the cross section of the cylinder $\omega_2$. We again  emphasize
 that in our case  the problem becomes independent of the associated
 problem \eqref{eq:54} on the cross section. Concerning results of the
 type considered here, it is worth mentioning the recent works
 \cite{ers, rrr} that generalize the results of \cite{pr}
 to $p$-Laplacian type operators, with $p>2$. Most of our techniques are variational, so it seems plausible that a generalization to systems of most or all of the results could be possible, but we have not pursued this issue in the present manuscript.
 \\[2mm]
{\bf Acknowledgments.} The authors thank the anonymous referee for his
or her helpful comments that improved the clarity of the manuscript. P.\,Roy is supported by the Core Research Grant (CRG/2022/007867) of SERB,
India and  Pradeep Jotwani young faculty fellowship from IIT Kanpur.
Part of this  was done  when P.\,Roy was visiting the
Technion. He thanks the Mathematics Department for
its invitation and hospitality.
The research of I.\,Shafrir was supported by funding from the Martin and Sima Jelin Chair in Mathematics.
\section{Notations, various auxilliary spaces and problems }\label{33}

a
The suitable space  for our problem  \eqref{eq:5} is
\begin{equation*}
V(\Omega_\ell) = \left\{ v \in H^1(\Omega_\ell) \ | \ v = 0  \text{ on
  }\gamma_\ell  \right\}~(\text{see }\eqref{eq:decom}),
\end{equation*}
 where the boundary condition should be interpreted in the sense of
 traces.  

 We denote by $B(X_0,K)\subset\R^m$ the ball of radius $K$ centered at
 $X_0$ (w.r.t.~the Euclidean norm) and by
  ${\widetilde B}_K(p)$ the ball  of center $p$ and radius $K$,
  w.r.t.~the Euclidean norm 
  in the $(m-1)$-dimensional $x'$-space, as defined at the beginning
  of the Introduction.
By $S^{m-1}$ we denote the unit sphere in $\R^m$. 

   Fix a point $P_0\in\partial\omega_1$. Since 
  $\partial\omega_1$ is assumed to be of class $C^1$, in a
  neighborhood of $P_0\in\partial\omega_1$ the boundary of $\omega_1$ can be described
  as a graph of a $C^1$ function. We may assume without loss of generality
  that the normal vector to $\partial\omega_1$ at $P_0$ is
  $\nu={\bf e}_1 =(1,0,\ldots,0)\in\R^m$, and by applying an appropriate shift we may further assume that the point
$P_0$ is the origin in ${\R}^{m}$.   
  There exist  $\delta>0$
  and a function $f=f^{P_0}\in C^1( {\widetilde B}_\delta(0))$ such
  that 
\begin{equation}\label{bet} \omega_1\cap B(P_0,\delta)=\{X=(x_1,x')\in
  B(P_0,\delta) \big|  \  x_1<f(x') \} .
\end{equation}
 Moreover, our assumption about the normal vector at $P_0$ implies
 that
 \begin{equation}
   \label{eq:2}
   \nabla_{x'} f(0)=(0,\ldots,0)\in{\R}^{m-1}.
 \end{equation}
Setting $f_\ell(x')=f_\ell^{P_0}(x')=\ell f(x'/\ell)$ we deduce from \eqref{bet} that 
\begin{equation}\label{eq:bet} 
\ell\omega_1\cap B(\ell P_0,\delta\ell)=\{X=(x_1,x')\in
  B(\ell P_0,\delta\ell) \big|  \  x_1<f_\ell(x') \} .
\end{equation}
 Note that a fixed value of  $\delta>0$ can be chosen such that
 \eqref{bet}--\eqref{eq:bet} hold for the local parameterizations of
 $\partial\omega_1$ near all the points
 $P_0\in\partial\omega_1$
 (after an appropriate change of variables). It then follows from
 \eqref{eq:2} that 
\begin{equation}\label{dsada}
        \|\grad_{x'} f_\ell\|_{L^\infty( \widetilde{B}_{\delta\ell}(0))}=o(\delta)\,,
      \end{equation}
 and an analogous estimate holds for the parameterizations around all
 points $P_0\in\partial\omega_1$. 
 
\smallskip

Next we introduce some sets and spaces  that will be useful in the sequel. Consider
$P_0\in\partial\omega_1$ with normal vector $\nu\in{
  S}^{m-1}$, and assume first that 
$P_0$ is the origin and that $\nu={\bf e}_1$.  For any  $0<K
\leq 2\ell^\beta$ define (under the above assumptions on $P_0$ and $\nu$):
\begin{equation}\label{kk1'}
\Omega_{\ell, K}^{P_0} := \left\{ X\in\ell\omega_1\ \big|  \
  f_\ell(x') -K < x_1 <  f_\ell(x^\prime),\ x' \in {\widetilde B}_K(0) \right\} \times \omega_2\,.
\end{equation}
Note that $\Omega_{\ell, K}^{P_0}\subset B(\ell P_0,\delta\ell)\times\omega_2$ for
sufficiently large $\ell$. In the general case when $\nu$ is any vector
in $S^{m-1}$ and $P_0$ is not necessarily the origin, we denote
by    $\Omega_{\ell, K}^{P_0}$ the image of a set of the form
\eqref{kk1'} by a rotation in the $X$-space 
that sends ${\bf e}_1$ to $\nu$ and an appropriate translation. In the
rest of the section $\nu$ denotes an arbitrary vector in $S^{m-1}$.
                                                                                                                                                                                                                                                                                                                                                                       
The following notation will also be useful when we take the limit
$\ell\to\infty$. For each $\nu\in{ S}^{m-1}$ set 
\begin{equation}\label{kk2}
 \mathcal{B}_{K}(\nu) :=\left\{ X\in\R^m\ \big|
  X\cdot\nu\in(-K,0)\text{ and }\|X-(X\cdot\nu)\nu\|<K\right\}\times\omega_2
\end{equation}and
\begin{equation*}
\label{poi}
\mathcal{B}_{\infty}(\nu) := \left\{ X\in\R^m\ \big|
  X\cdot\nu\in(-\infty,0)\right\}\times\omega_2\,.
\end{equation*}
With the sets defined above we associate  the following spaces:
\begin{equation*}
\label{kk3}
V \big(\Omega_{\ell, K}^{P_0}\big) := \left\{ u \in H^1(\Omega_{\ell, K}^{P_0})
  \ \big| \ u = 0  \  \text{ on }  \del\Omega_{\ell, K}^{P_0}\setminus
 \Gamma_\ell \right\},
\end{equation*}
\begin{equation*}
\label{kk4}
V\left(\mathcal{B}_{K}(\nu)\right):= \left\{ u \in
  H^1(\mathcal{B}_{K}(\nu)) \ \big|  \ u = 0  \ \textrm{on}  \ \del
  \mathcal{B}_{K}(\nu)\setminus \big(\{X\in\R^m\,|\,X\cdot\nu=0\}\times\omega_2\big)\right\}
\end{equation*}
and
\begin{equation*}
\label{ghv}
V\left(\mathcal{B}_{\infty}(\nu)\right) :=  \left\{ u\in
  H^1\left(\mathcal{B}_{\infty}(\nu) \right)  \  \big| \  u= 0  \
  \textrm{on}
\ \del
  \mathcal{B}_{\infty}(\nu)\setminus \big(\{X\in\R^m\,|\,X\cdot\nu=0\}\times\omega_2\big)\right\}.  
\end{equation*}
\smallskip

For each $\nu \in {S}^{m-1}$ we define
\begin{align}
\label{eq:s_K}
s_K^\nu= \inf_{u \in V(\mathcal{B}_{K}(\nu))}\frac{\int_{\mathcal{B}_{K}(\nu)} (A \grad u)\cdot \grad u }{ \int_{\mathcal{B}_{K}(\nu)} u^2}\\
\intertext{ and }
\label{t_i}
\widetilde{Z}^\nu = \inf_{u \in V(\mathcal{B}_{\infty}(\nu))}\frac{\int_{\mathcal{B}_{\infty}(\nu)} (A \grad u)\cdot \grad u }{ \int_{\mathcal{B}_{\infty}(\nu)} u^2}\,.
\end{align}
Similarly, for each $P_0\in\partial\omega_1$ with normal vector
$\nu\in{S}^{m-1}$ as above we set
\begin{equation}
\label{eq:s_K-Om}
s_{\ell,K}^{P_0}= \inf_{u \in
  V(\Omega^{P_0}_{\ell,K})}\frac{\int_{\Omega^{P_0}_{\ell,K}} (A \grad
  u)\cdot \grad u }{ \int_{\Omega^{P_0}_{\ell,K}} u^2}\,.~
\end{equation}
Note that $s_{\ell,K}^{P_0}$ is attained by a unique positive
function $w_{\ell,K}$ satisfying $\int_{\Omega_{\ell, K}^{P_0}} w_{\ell,K}^2=1$ and 
\begin{equation}
\label{eq:e22}
\left\{
\begin{aligned}
  -&\textrm{div}(A \nabla w_{\ell,K})=s_{\ell,K}^{P_0} w_{\ell,K}  &\text{ in }  \Omega_{\ell, K}^{P_0} ,\\
  & (A \nabla w_{\ell, K})\cdot\tilde{\nu}=0 &\text{ on }     \partial\Omega_{\ell, K}^{P_0} \cap \Gamma_\ell, \\
     &w_{\ell, K}=0  &\text{ on } \del\Omega_{\ell, K}^{P_0}\setminus
 \Gamma_\ell.
\end{aligned}
\right.
\end{equation}
 In \eqref{eq:e22} we denoted by $\tilde\nu$  the unit exterior normal
 on $\partial\Omega_{\ell, K}^{P_0} \cap \Gamma_\ell$.
 We conclude this section with a simple and useful property of $Z^\nu$:

 \begin{lemma}
   \label{lem:Z-nu}
   The map $\nu \mapsto Z^\nu$ is
continuous on $S^{m-1}$.
\end{lemma}
\begin{proof}
  It is clear from the definition \eqref{eq:4} of $A_\nu$ that we have
  $$
\sup_{\xi\in\omega_2}\|A_{\nu_1}(\xi)-A_{\nu_2}(\xi)\|_{L^\infty}\le
C|\nu_1-\nu_2|,\quad\forall\nu_1,\nu_2\in S^{m-1}.
  $$
  The result then follows from \eqref{xc},
  since if $u$ is any admissible function in \eqref{xc}, 
  then
  $$
  \Big|\int_{(-\infty,0)\times\omega_2} (A_{\nu_1}\nabla u)\cdot
  \nabla u-\int_{(-\infty,0)\times\omega_2} (A_{\nu_2}\nabla u)\cdot
  \nabla u\Big|\le
 C |\nu_1-\nu_2|\int_{(-\infty,0)\times\omega_2}|\nabla u|^2.
  $$
\end{proof}


\section{The upper bound construction}\label{sec2}

\noi For each $\nu\in S^{m-1}$ let the matrix $A_\nu$
be given by \eqref{eq:4}. We consider on $(-\ell, 0)\times \omega_2
\subset \R^{1+p}$ the  first eigenvalue and the corresponding (unique) 
normalized positive  
eigenfunction, for the following  
problem:
\begin{equation}
\label{eq:6}
\left\{
\begin{aligned}
  -&\textrm{div}(A_{\nu} \nabla v_{\ell}^\nu)= Z_{\ell}^\nu v_{\ell}^\nu  \quad\text{ in }   (-\ell, 0) \times \omega_2,\\
    & v_\ell^\nu=0 \quad\text{ on } \left((-\ell, 0) \times \del\omega_2 \right)\cup \left(\{ -\ell \} \times \omega_2\right), \\
    & (A_\nu \nabla v_\ell^\nu)\cdot{\bf e_1}=0 \quad\text{ on } \{ 0 \} \times \omega_2,\\
    & v_\ell^\nu>0 \text{ in }(-\ell, 0) \times \omega_2\text{ and }\int_{(-\ell, 0) \times \omega_2} | v_\ell^\nu |^2 =1.
\end{aligned}
\right.
\end{equation}
Here ${\bf e_1}$ denotes the unit vector in the direction of the
$x_1$-coordinate.

  We recall, see   \cite[Lemma 5.1 and Lemma 5.2]{pr}, that  \begin{equation}\label{xc1}
Z^\nu = \lim_{\ell \rightarrow \infty} Z_\ell^\nu\le\mu_1\,.
        \end{equation}

\begin{theorem}\emph{[Upper Bound]}\label{upperbound}
 We have
$$\limsup_{\ell \rightarrow \infty} \lambda_\ell  \leq     \inf_{\nu  \in S^{m-1}} Z^\nu. $$
\end{theorem}

\begin{proof}
Let $\nu \in S^{m-1}$ be fixed.  Since we assumed that $\del\omega_1$
is of class $C^1$ there exists a point $P_0\in \del\omega_1$ such that
the outward unit  normal  at $P_0$ is  $ \nu$.
\smallskip

In order to change variables in such a way that $\nu$ will be
transformed 
to ${\bf e}_1$, we first choose $m-1$ vectors $p_2,  \ p_3 ,  \ldots ,  p_m  \in \R^m$ that
together with $\nu$ form  an orthonormal basis of $\R^m$.
Let $B$ denote the orthogonal  matrix whose rows are
$\nu,p_2,\ldots,p_m$. Therefore  its transpose $B^T$ is given by
$$B^T = \begin{pmatrix}
  \nu^T  & p_2^T  & .  &  . & . &  p_m^T
\end{pmatrix}.$$
We change variables from $x=(X,\xi)$ to $y=(Y,\xi)$ by setting  $Y =
BX$. Let $v(Y,\xi) = u(X,\xi)$ and denote  $Y = (y_1, \ldots  ,
y_m)=(y_1,y')$.    An easy  computation  gives  that $$\grad_X u =  B^T\grad_Y v.$$
For ease of notation we  will continue to denote by
$\Omega_\ell$ and $\omega_1$ the same domains as above also when they are represented  in the new variables.  A direct
computation yields
\begin{equation}\label{eq:po}
\int_{\Omega_\ell} (A\grad_x u)\cdot \grad_x u  = \int_{\Omega_\ell}
(A^B\grad_y v) \cdot \grad_y v   \    \text{ and }  \ \int_{\Omega_\ell} u^2 = \int_{\Omega_\ell} v^2,
\end{equation}
where $$A^B = \begin{pmatrix}
BA_{11}B^T & BA_{12}\\
\left( BA_{12} \right)^T & A_{22}
\end{pmatrix}.$$

Let $\Omega_{\ell, K}=\Omega_{\ell, K}^{P_0}$ be as in
Section~\ref{33}, but defined using the new variables, i.e., 
\begin{equation}\label{kk1'y}
\Omega_{\ell, K}:= \left\{ Y\in\ell\omega_1\ \big|  \ f_\ell(y') -K <
  y_1 <  f_\ell(y'),\ y' \in {\widetilde B}_K(0) \right\} \times \omega_2\,.
\end{equation}
Fix any $\Phi\in C^\infty_c(\R)$ with $\supp\Phi\subset(-1/\sqrt{m-1},1/\sqrt{m-1})$
satisfying 
\begin{equation}
  \label{eq:8}
  \int_{-\infty}^{\infty}\Phi^2(t)\,dt =1\,.
\end{equation}
Set
 $c_0:=\int_{-\infty}^{\infty}(\Phi^\prime)^2$. For any $K>0$ let
 $\Phi_K(t)=(1/\sqrt{K})\Phi(t/K)$. Then,
\begin{equation}
\label{0}
\int_{-\infty}^{\infty}\Phi_K^2 =1 \text{ and }    \int_{-\infty}^{\infty}|\Phi'_K|^2 =\frac{c_0}{K^2}.
\end{equation}
To simplify notations we will
denote in the sequel,
$$
\phi_j(y_j)=\Phi_K(y_j)\,,\quad j=2,\ldots,m\,.
$$
Define the function
\begin{equation*}
\label{sdf}
q_\ell(Y,\xi) =
\begin{cases}
v_K^\nu(y_1-f_\ell(y'), \ \xi) \Prod_{j=2}^m\phi_j(y_j)  & \text{
  on }   \ \Omega_{\ell, K},\\
0 &\text{ on }   \  \Omega_\ell \setminus \Omega_{\ell, K},
\end{cases}
\end{equation*}
 with $v_K^\nu$ given by \eqref{eq:6} (for $\ell=K$). Above we used the fact that
 $$
 \supp \Big(v_K^\nu(y_1-f_\ell(y'), \ \xi)
 \Prod_{j=2}^m\phi_j(y_j)\Big)\subset{\overline\Omega_{\ell, K}}\,,
 $$
 since
 \begin{equation*}
   \supp\big(\Prod_{j=2}^m\phi_j(y_j)\big)\subset
   \underbrace{(-\frac{K}{\sqrt{m-1}}, \frac{K}{\sqrt{m-1}})\times\cdots\times
     (-\frac{K}{\sqrt{m-1}}, \frac{K}{\sqrt{m-1}})}_\text{$m-1$
     times}\subset {\widetilde B}_K(0)\,,
 \end{equation*}
 thanks to the properties of $\Phi_K$.
Next we introduce another change of variables in order to  flatten the
part of the boundary where a Neumann condition is imposed. For that
matter we let $Z=(z_1,z')$ with
$$
z'=y'\text{ and } 
z_1 = y_1 -f_\ell (y')
,$$
and then set $z=(Z,\xi)$, for $Y=(y_1,y')\in \Omega_{\ell, K}$ .
The determinant of the Jacobian matrix  for the above change of
variables is $1$.
The change of variables from $x$ to $z$ transforms $\Omega_{\ell, K}$
to the following domain, which is independent of $\ell$:
\begin{equation}
  \label{eq:111}
  \widetilde{\Omega}_{K} := \mathcal{B}_{K}({\bf e}_1)= (-K,0)\times
  {\widetilde B}_K(0)\times \omega_2\quad(\text{see }\eqref{kk2}).
\end{equation}

Set
\begin{equation}
  \label{eq:13}
 g_K(z)=g_K(Z,\xi)= v_K^\nu(z_1, \ \xi)
 \Prod_{j=2}^m\phi_j(z_j) ~\text{ on }\widetilde{\Omega}_{K}\,.
\end{equation}

An easy  computation  gives
\begin{equation}
  \label{eq:1}
  \grad_{Y}q_\ell = \grad_{Z}  g_K +( \del_{z_1}g_K )
  B_\ell~\text { on }\Omega_{\ell, K}. 
\end{equation}
where
\begin{equation}
  \label{eq:37}
  B_\ell = -(0,
\del_{y_2}f_\ell,  \ldots ,
\del_{y_{m}}f_\ell )\,.
\end{equation}
By \eqref{eq:6} and \eqref{0} we get
\begin{equation}
  \label{eq:14}
  \int_{\Omega_{\ell}} q_\ell^2 =\int_{\Omega_{\ell,K}} q_\ell^2 =\int_{\widetilde{\Omega}_{K}} g_K^2\,dz=
\Big(\int_{(-K, 0)\times \omega_2} \!\!\!\!\!|v_K^\nu|^2\,dz_1d\xi\Big)
\Big(\int_{{\widetilde B}_K(0)}\prod_{j=2}^m \phi^2_j(z_j)\,dz'\Big)=1\,.
\end{equation}
Moreover, from \eqref{eq:1} we deduce that 
\begin{equation}\label{nam}
\int_{\Omega_{\ell}} (A^B\grad_y q_\ell )\cdot \grad_y q_\ell =
\int_{\Omega_{\ell, K}} (A^B\grad_y q_\ell )\cdot \grad_y q_\ell =
\int_{{\widetilde\Omega}_K} (A^B\grad_z
g_K)\cdot \grad_z g_K +I_1,  
\end{equation}
where $I_1$ satisfies
\begin{equation}
  \label{eq:9}
|I_1|\le C\|B_\ell\|_\infty \int_{{\widetilde\Omega}_K} |\grad g_K |^2 =
C\|B_\ell\|_\infty \int_{{\widetilde\Omega}_K}
\bigg|\grad \Big(v_K^\nu\prod_{j=2}^m
    \phi_j\Big)\bigg|^2\,dz \le C\|B_\ell\|_\infty =o({K/\ell})\,,
\end{equation}
by \eqref{dsada}. 
 \par Next we estimate  the integral $\int_{{\widetilde\Omega}_K} (A^B\grad_z
g_K)\cdot \grad_z g_K$ on the R.H.S.\,of \eqref{nam}. 
\begin{multline}
  \label{eq:21}
\int_{{\widetilde\Omega}_K} (A^B\grad_z
g_K)\cdot \grad_z g_K  =  \int_{{\widetilde\Omega}_K} (BA_{11}B^T\grad_Z g_K )\cdot \grad_Z g_K  \\+ 2 \int_{{\widetilde\Omega}_K} (BA_{12}\grad_\xi g_K) \cdot \grad_Z g_K  + \int_{{\widetilde\Omega}_K} (A_{22}\grad_\xi g_K) \cdot \grad_\xi g_K :=  J_1+ 2 J_2 +J_3.
\end{multline}
Let us first calculate the term $J_1 $. 
\begin{equation}
  \label{eq:18}
\begin{aligned}
J_1 &= \int_{{\widetilde\Omega}_K}\big(\nu^TA_{11}\nu\big)
|\partial_{z_1}v_K^\nu|^2
\Big(\prod_{i=2}^m\phi_i^2\Big) \\
&+2\sum_{k=2}^m\int_{{\widetilde\Omega}_K} \big( \nu^TA_{11} p_k\big)v_K^\nu
(\partial_{z_1}v_K^\nu)\phi_k\phi'_k\Big( \prod_{\substack{i=2\\ i\neq
    k}}^m\phi_i^2 \Big)+ \sum_{k=2}^m \int_{{\widetilde\Omega}_K}\big(p_k^TA_{11}
p_k\big) |v_K^\nu|^2 |\phi'_k|^2
\Big(\prod_{\substack{i=2\\\ i\neq k}}^m\phi_i^2\Big)\\ &+
\sum_{\substack{2\le k, j\le m\\ k\ne j}}^m \int_{{\widetilde\Omega}_K} \big(p_k^TA_{11}p_j\big)|v_K^\nu|^2 \phi'_k \phi_k \phi'_j\phi_j \Big(\prod_{\substack{i=2\\ i\neq k, j}}^m\phi_i^2\Big).
\end{aligned}
\end{equation}
Notice that the second and last terms on the R.H.S.~of
\eqref{eq:18}  vanish since 
$$
\int_{-K}^K\phi_j(t) \phi'_j(t)\,dt= 0,~\forall j\ge
2\,.
$$
Using \eqref{eq:6} and \eqref{0} in \eqref{eq:18} yields
\begin{equation}
\label{ww}
J_1 \leq  \int_{(-K,0) \times \omega_2}(\nu^TA_{11}\nu) |\del_{z_1}v_K^\nu |^2\,dz_1d\xi+\frac{C}{K^2}\,.
\end{equation}
Clearly
\begin{equation}
\label{j3}
J_3 =\int_{(-K, 0) \times \omega_2} (A_{22}\grad_\xi v_K^\nu)\cdot \grad_\xi v_K^\nu  .
\end{equation}
Finally we turn to $J_2$. 
\begin{equation}
\label{j2}
\begin{aligned}
J_2 &=  \int_{{\widetilde\Omega}_K} \big((A_{12}\grad_{\xi}g_K) \cdot \nu\big) \del_{z_1}g_K  + \sum_{k=2}^m\int_{{\widetilde\Omega}_K} \big((A_{12}\grad_{\xi}g_K)\cdot p_k\big) \del_{z_k}g_K \\
& =
\int_{(-K,0)\times\omega_2}\big((A_{12}\grad_{\xi}v_K^\nu)\cdot
\nu\big) \partial_{z_1}v_K^\nu +  \sum_{k=2}^m\int_{{\widetilde\Omega}_K}
\big((A_{12}\grad_{\xi}v_K^\nu)\cdot p_k\Big)
v_K^\nu \phi_k \phi'_k \Big(\prod_{i=2, i \neq k}^m \phi_i^2\Big)\\
& =\int_{(-K,0)\times\omega_2}\big((A_{12}\grad_{\xi}v_K^\nu)\cdot
\nu\big) \partial_{z_1}v_K^\nu\,.
\end{aligned}
\end{equation}
From \eqref{eq:21}, \eqref{ww}--\eqref{j2} and \eqref{eq:6} we
deduce that
\begin{equation}
  \label{eq:20}
  \begin{aligned}
  \int_{{\widetilde\Omega}_K} (A^B\grad_zg_K)\cdot \grad_z g_K&\le
\int_{(-K, 0) \times
  \omega_2}(\nu^TA_{11}\nu)|\del_{z_1}v_K^\nu |^2+2 \int_{(-K, 0) \times
  \omega_2} \big((A_{12}\grad_{\xi}v_K^\nu) \cdot\nu
\big)\del_{z_1}v_K^\nu \\
&+\int_{(-K, 0) \times \omega_2}(A_{22}\grad_\xi
v_K^\nu)\cdot \grad_\xi v_K^\nu +\frac{C}{K^2}\\
&= \int_{(-K, 0) \times \omega_2} (A_\nu\nabla v_K^\nu)\cdot\nabla v_K^\nu +\frac{C}{K^2}=Z_K^\nu +\frac{C}{K^2}\,.
\end{aligned}
\end{equation}
Combining \eqref{eq:19} and \eqref{eq:po} with \eqref{eq:20} and 
\eqref{eq:14}--\eqref{eq:9} we obtain
\begin{equation}
  \label{eq:22}
  \lambda_\ell\leq \int_{\Omega_{\ell}} (A^B\grad_y q_\ell )\cdot
    \grad_y q_\ell \le Z_K^\nu +\frac{C}{K^2}
  +o({K/\ell})
\,.
\end{equation}
For any sequence
satisfying $\lim_{j\to\infty}\ell_j=\infty$, we choose
$K=\ell_j^\beta$ and pass to the limit in \eqref{eq:22} to deduce,
using \eqref{xc1}, that
\begin{equation}
  \label{eq:23}
  \limsup_{\ell\rightarrow \infty}\lambda_\ell\le \lim_{\ell\to\infty}
  Z_{\ell^\beta}^\nu =Z^\nu\,.
\end{equation}
 The result follows from \eqref{eq:23} since the direction $\nu$ can be chosen arbitrarily.
\end{proof}

\section{The Lower bound}\label{sec44}

The following  theorem is  the  main  result  of this section:
 \begin{theorem}\label{lb}
 \emph{[Lower bound]} Assume that $A_{12} \grad_{\xi} W \not\equiv 0  \ a.e.  \ \textrm{in}  \ \omega_2$, then
 \begin{equation}\label{bn}
 \liminf_{\ell \rightarrow \infty} \lam_\ell \geq \inf_{\nu \in S^{m-1}}Z^{\nu}.
 \end{equation}  
\end{theorem}
\begin{remark}
  Thanks to \rlemma{lem:Z-nu} actually $\inf_{\nu \in S^{m-1}}Z^{\nu}=\min_{\nu \in S^{m-1}}Z^{\nu}$.
\end{remark}
The proof of \rth{lb} requires several preliminary results. We start with the next Lemma.
\begin{lemma}\label{l5}
Let $\nu \in S^{m-1}$ be such that  the strict inequality $Z^\nu <
\mu_1$ holds. Then, $Z^\nu = \widetilde{Z}^\nu$, where
$\widetilde{Z}^\nu$ is defined in \eqref{t_i}.
\end{lemma}
\begin{proof}
 Without loss of generality we assume again that $\nu={\bf e}_1$.
 It  is easy  to show  that
 \begin{equation}\label{eq:v}
 \widetilde{Z}^{\nu} \leq Z_\ell^{\nu} ,\quad\forall\ell>0.
\end{equation}
Indeed, to prove \eqref{eq:v} it suffices to use  in \eqref{t_i}  a test function of the
form
\begin{equation*}
  w_{\ell}(X,\xi)=\begin{cases}
     \Big(\Prod_{j=2}^m\Phi_K(x_j)\Big)v_\ell^{\nu}(x_1,\xi) &
     x_1\in(-\ell,0)\\
     0 & \text{ otherwise}
   \end{cases}
   ,
\end{equation*}
for any $K>0$, with $\Phi_K$ as defined in the previous section and
$v_\ell^\nu$ as given by \eqref{eq:6}. A similar computation to the
one used in the
proof of \rth{upperbound} gives
\begin{equation}
  \label{eq:10}
   \widetilde{Z}^{\nu}\le Z_\ell^{\nu} +\frac{C}{K^2}\,.
 \end{equation}
Actually the computation here is even simpler since we do not need to
 flatten the boundary. Letting $K$ go to infinity in \eqref{eq:10} yields
 \eqref{eq:v}. Finally, passing to the limit $\ell\to\infty$ in
 \eqref{eq:v}, taking into account \eqref{xc1}, 
 we deduce the inequality
 \begin{equation}
   \label{eq:11}
   \widetilde{Z}^{\nu}\le Z^{\nu}\,.   
 \end{equation}

For  the reverse inequality, we assume as above w.l.o.g.~that
$~\nu={\bf e}_1$ and  notice that our assumption $Z^\nu
< \mu_1$ implies, thanks to \cite[Prop.~6.1]{pr}, that
$Z^{\nu}$ is attained by some function $v^{\nu}$ that satisfies the
Euler-Lagrange equation
\begin{equation}
  \label{eq:12}
  \left\{
    \begin{aligned}
  -\div(A_\nu\nabla v^{\nu})&=Z^{\nu} v^{\nu}&\text{ on }(-\infty, 0)\times
  \omega_2,\\
  v^{\nu}&=0 &\text{ on }(-\infty, 0)\times
   \partial\omega_2,\\
  (A_\nu\nabla v^{\nu})\cdot{\bf e}_1&=0 &\text{ on }\{0\}\times\omega_2.
\end{aligned}
\right.
\end{equation}
By elliptic regularity and the strong maximum principle we
know that $v^{\nu}$ is continuous and can be assumed to be positive in $(-\infty, 0)\times
\omega_2$. We extend $v^{\nu}$ naturally  to
$\mathcal{B}_{\infty}(\nu)$ by setting
\begin{equation*}
  \tilde v^{\nu}(x_1,x',\xi)=v^{\nu}(x_1,\xi)\,,~x_1\in(-\infty,0),\,x'\in\R^{m-1},\,\xi\in\omega_2.
\end{equation*}
It is easy to verify that \eqref{eq:12} implies that $\tilde v^{\nu}$
satisfies
\begin{equation*}
  -\div(A\nabla \tilde v^{\nu})=Z^{\nu} \tilde v^{\nu}\text{ on }\mathcal{B}_{\infty}(\nu).
\end{equation*}
Next, we recall the  following version of  the Picone  identity,
\begin{equation}
  \label{eq:55}
   (A\nabla u)\cdot \nabla u-(A\nabla
  v)\cdot \nabla\big(\frac{u^2}{v}\big)  =  A\big(\nabla u-\frac{u}{v}\nabla
  v\big)\cdot \big(\nabla u-\frac{u}{v}\nabla v\big)  \geq 0\,.
\end{equation}

 Take any  $u\in V(\mathcal{B}_{\infty}(\nu))$ (that we may assume to
 be smooth). By
\eqref{eq:55}, integration by parts  and \eqref{eq:12} we obtain
 \begin{equation*}
   \label{eq:64}
\begin{aligned}
   0&\leq\int_{\mathcal{B}_{\infty}(\nu)}   A\big(\nabla
   u-\frac{u}{\tilde v^{\nu}}\nabla
  \tilde v^{\nu}\big) \cdot \big(\nabla u-\frac{u}{\tilde
    v^{\nu}}\nabla \tilde v^{\nu}\big)   \\&=
\int_{\mathcal{B}_{\infty}(\nu)} (A\nabla u) \cdot\nabla u  - (A\nabla
   \tilde v^{\nu}) \cdot \nabla\big(\frac{u^2}{\tilde v^{\nu}}\big)\\&=
 \int_{\mathcal{B}_{\infty}(\nu)} (A\nabla u)\cdot\nabla u
+\int_{\mathcal{B}_{\infty}(\nu)}
 \textrm{div}(A\nabla \tilde v^{\nu})\big(\frac{u^2}{\tilde v^{\nu}}\big)-\int_{\{0\}\times \R^{m-1}\times \omega_2}
\big( (A\nabla \tilde v^{\nu})\cdot\nu_1\big) \big(\frac{u^2}{\tilde v^{\nu}}\big)\\
&=\int_{\mathcal{B}_{\infty}(\nu)}(A\nabla u)\cdot \nabla u
- Z^{\nu} u^2,
\end{aligned}
\end{equation*}
where $\nu_1= (1, 0,  \ldots , 0) \in \R^{m+p-1}$. Since the above
holds for an arbitrary $u\in V(\mathcal{B}_{\infty}(\nu))$, we deduce
the desired  inequality $\widetilde{Z}^{\nu}\ge Z^{\nu}$.
\end{proof}

\begin{lemma}
  \label{lem:Ktoi}
  For each $\nu\in{S}^{m-1}$ we have
  \begin{equation}
    \label{eq:25}
  s_K^\nu  \searrow \widetilde{Z}^\nu\text{ as }K\nearrow\infty.
  \end{equation}
\end{lemma}

\begin{proof}
  First we notice that for any $K_1<K_2$ we have $s_{K_2}^\nu\le
  s_{K_1}^\nu$. This follows from the fact that every $u\in
  V(\mathcal{B}_{K_1}(\nu))$ can be extended to a function $\tilde u\in
  V(\mathcal{B}_{K_2}(\nu))$ by setting $\tilde u=0$ on
  $\mathcal{B}_{K_2}(\nu)\setminus \mathcal{B}_{K_1}(\nu)$. This
  implies that the function $K\mapsto  s_K^\nu$ is non-increasing,
  whence the limit $\lim_{K\to\infty} s_K^\nu$ exists. To identify the
  value of the limit as  $\widetilde{Z}^\nu$ it suffices to note that
  $\bigcup_{K>0} V({\mathcal B}_K(\nu))$ is dense in $V({\mathcal B}_\infty(\nu))$.
\end{proof}

\begin{lemma}\label{conver}
 For any
  point $P_0\in\partial\omega_1$ with normal vector $\nu$  we have
  \begin{equation}
    \label{eq:31}
    |s_{\ell,K}^{P_0}-s_K^\nu|=o(K/\ell)\,,
  \end{equation}
   for all $K\in(0,\ell)$ and the estimate holds uniformly for all
   points $P_0\in\partial\omega_1$.
\end{lemma}
\begin{proof}
  As above we may assume that $\nu={\bf e}_1$ and then
  $\Omega_{K,\ell}^{P_0}$ is given by \eqref{kk1'}. Consider any $v\in
  V(\Omega_{K,\ell}^{P_0} )$ satisfying
  \begin{equation}
    \label{eq:32}
    \int_{\Omega_{K,\ell}^{P_0}} |v|^2=1\,.
  \end{equation}
We change variables by
\begin{equation*}
\left\{
\begin{aligned}
&z_1 = x_1 -f_\ell (x')\,, \\
&z_i = x_i \hspace{3mm} \textrm{for} \ i = 2,  \ldots ,   m\,,
\end{aligned}
\right.
\end{equation*}
and define  $w(z)=w(Z, \xi)= v(X,\xi) $ where $Z= (z_1, z_2, \ldots ,
z_m)$ and $(z_{m+1},\ldots,z_{m+p})=(\xi_1,\ldots,\xi_p)$. Then $w \in V(\mathcal{B}_{K}(\nu))$ and by the same
computations as in the proof of \rth{upperbound}  it  is easy to check
that
\begin{equation}
  \label{eq:38}
  \int_{\mathcal{B}_{K}(\nu)}|w|^2\,dz= \int_{\Omega_{K,\ell}^{P_0}} |v|^2\,dx=1,
\end{equation}
and
\begin{equation}
  \label{eq:36}
  \int_{\Omega_{K,\ell}^{P_0}}(A\nabla_xv)\cdot\nabla_xv\,dx=\int_{\mathcal{B}_{K}(\nu)}(A\nabla_zw)\cdot\nabla_zw\,dz+I_1\,,
\end{equation}
where $I_1$ satisfies
\begin{equation}
  \label{eq:33}
  |I_1|\le C\|B_\ell\|_\infty
  \int_{\mathcal{B}_{K}(\nu)}|\nabla_zw|^2\le o(K/\ell) \int_{\mathcal{B}_{K}(\nu)}|\nabla_zw|^2\,,
\end{equation}
 with $B_\ell$ as in \eqref{eq:37}.
By  \eqref{eq:38} and \eqref{eq:36}--\eqref{eq:33} we deduce that
\begin{equation*}
  s_{\ell,K}^{P_0}\leq (1+o(K/\ell)) s_K^\nu\leq s_K^\nu+o(K/\ell))\,.
\end{equation*}
The reverse inequality is proved similarly and we are led to \eqref{eq:31}.
\end{proof}

Next we obtain some decay estimates for $u_\ell$ away from the
boundary (these are analogous to the estimates in \cite[Thm~5.1]{pr}).

\begin{lemma}\emph{[Asymptotic of first eigenfunction]}
\label{ef}
Assume that \eqref{co} holds and let $r \in  (0, \ell -1]$. Then,
there exist  constants $\alpha \in (0,1)$ and  $C> 0$ and  for
sufficiently  large $\ell,$  we have
\begin{equation}
  \label{eq:24}
 \int_{\Omega_r}| u_\ell|^2 \leq \alpha^{[\ell -r]} \text{ and
 }\int_{\Omega_r}| \grad  u_\ell |^2  \leq  C\alpha^{[\ell -r]}.
 \end{equation}
\end{lemma}
\begin{proof}
  Let $K$ be positive integer such that $K+1 < \ell$ and
 let $\rho^K =\rho^K(X)$ 
 be a Lipschitz  continuous  function on $\R^m$  such that
 $\rho^K = 1$ on $B(0,K)$ and  $\rho^K = 0$
 outside  $B(0,K+1)$.  We may also  assume that $|\grad\rho^K | \leq C$ for  some  positive  constant $C>0$.
 Testing the  Euler-Lagrange  equation \eqref{eq:5}
 satisfied by $u_\ell$ (with $\sigma=\lambda_\ell$) with $|\rho^K|^2u_\ell$ , we get
\begin{equation*}
\int_{\Omega_\ell}  A\nabla u_\ell\cdot \nabla ( |\rho^K|^2u_\ell)= \lambda_\ell \int_{\Omega_\ell}|\rho^K|^2u_\ell^2\,,
\end{equation*}
i.e.,
\begin{equation} \label{lla}
 \int_{\Omega_\ell} \big(A \nabla (|\rho^{K}|u_\ell)\big)\cdot\nabla (|\rho^K|u_\ell) -
\int_{\Omega_\ell}u_\ell^2 (A\nabla |\rho^K|)\cdot\nabla | \rho^K| = \lambda_\ell \int_{\Omega_\ell}|\rho^K|^2u_\ell^2\,.
\end{equation}
Since $|\rho^K|^2 u_\ell \in H_0^1(\Omega_\ell) $, we get from
\eqref{eq:26} that
\begin{equation}\label{diria}
\sigma_\ell^1  \int_{\Omega_\ell}u_\ell^2 |\rho^K|^2 \leq
\int_{\Omega_\ell} A\nabla(\rho^K u_\ell)\cdot \nabla (\rho^K  u_\ell) \,.
\end{equation}
Combining \eqref{lla}--\eqref{diria}, we get
\begin{equation*}\label{zero1a}
\begin{aligned}
(\sigma_\ell^1  -\lambda_\ell ) \int_{\Omega_{K+1}} u_\ell^2 |\rho^K|^2 \leq
\int_{\Omega_{K+1}}u_\ell^2 (A\nabla \rho^K)\cdot \nabla \rho^K &= \int_{
  \Omega_{K+1} \setminus \Omega_K }u_\ell^2 (A\nabla
\rho^{K})\cdot \grad \rho^{K}\\
&\leq C_A  \int_{  \Omega_{K+1} \setminus \Omega_K }u_\ell^2\,.
\end{aligned}
\end{equation*}
By \eqref{co} there exists a  direction  $\nu_1 \in S^{m-1}$ such that
$A_{12} \grad_\xi W \cdot \nu_1 \neq 0$. By Theorem \ref{crs}
it follows that
$$
\inf_{\nu\in S^{m-1}}Z^\nu\le \min\{Z^{\nu_1},Z^{-\nu_1}\} < \mu_1
\,.
$$
Combining this  with Theorem \ref{chipot}  and Theorem
\ref{upperbound}, we obtain that for sufficiently large $\ell$ we have
$\sigma_\ell^1 - \lam_\ell > \mu >0$, for some constant $\mu$.
Therefore we have
$$\int_{\Omega_K}u_\ell^2 \leq \frac{C_A}{\mu + C_A}\int_{\Omega_{K+1}}u_\ell^2.$$
Iterating this  formula from $K = r, r+1, \ldots ,[\ell]$ we get
the first inequality in \eqref{eq:24} with $\alpha:=\frac{C_A}{\mu +
  C_A}$.
Finally, the second estimate in \eqref{eq:24} follows from the first
one by a similar argument to the one used in the proof of \cite[Thm 5.2]{pr}.
\end{proof}
\\[2mm]

The next Lemma provides the last ingredient needed to the proof of
\rth{lb}. The main assumption \eqref{contr} below is used in the proof by
contradiction of that Theorem.
\begin{lemma}\label{nm}
 Assume that
\begin{equation} \label{contr}
 \liminf_{\ell \rightarrow \infty} \lambda_\ell =\lim_{k \rightarrow \infty} \
\lambda_{\ell_k}< \min_{\nu \in S^{m-1}} \widetilde{Z}^\nu.
 \end{equation}
Let $\beta\in(0,1)$.  Then 
\begin{equation}\label{po}
\lim_{k\to\infty} \int_{ \Omega_{\ell_k} \setminus \Omega_{\ell_k -\ell_k^\beta}} u_
{\ell_k}^2 =0.
\end{equation}
\end{lemma}

\begin{proof}
For the sake of simplicity  we shall denote in the sequel $\ell$
instead of $\ell_k$.
The main step of the proof consists of establish the following
estimate
\begin{equation}
  \label{eq:estimate}
  \int_{\Omega_{\ell,\ell^\beta}^{P_0}}u_\ell^2\le
  Ce^{[\ell^\beta]\log\alpha},\;\forall P_0\in\partial\omega_1,
\end{equation}
for some $\alpha\in(0,1)$.
We fix a point $P_0\in\partial\omega_1$, and we assume w.l.o.g (as in
Section\,\ref{33}) 
that the exterior normal at $P_0$ is $\nu={\bf e_1}$. In what follows
we shall use the shorthand notation 
 $\Omega_{\ell,K}=\Omega_{\ell,K}^{P_0}$.

 Let $K\in[\ell^\beta, 2\ell^\beta]$. Define the function $\rho_1^K =
 \rho_1^K(x')$  on ${\R}^{m-1}$ to be a Lipschitz  continuous  function satisfying
 $$
 0\le\rho_1^K\le1,\,  \rho_1^K = 1 \text{ on }{\widetilde B}_{K-1}(0)
 \text{ and }\rho_1^K =0  \text{ outside } {\widetilde B}_{K}(0).
 $$
  We also  assume that $|\grad_{x'}  \rho_1^K | \leq C$ for  some  positive  constant $C>0$. Let $s : \R \rightarrow \R$ be a Lipschitz continuous  function  such that
 \begin{equation*}
s(x_1)  =
\left\{
\begin{aligned}
&0 \hspace{7mm} \textrm{for} \ x_1 \leq 0,\\
&x_1 \hspace{5mm} \textrm{for} \ 0 \leq x_1 \leq 1,\\
&1 \hspace{7mm} \textrm{for} \ x_1 \geq 0
\end{aligned}
\right.
\end{equation*}
and
define  the function $g_\ell^K  =  u_\ell |\rho_{K}|^2\in V(\Omega_\ell) $  where  $\rho_{K}(x) = \rho_1^K s(x_1 - f_\ell(x) +K)$. Then  it is easy   to see   that
$\rho_K = 1$ on $\Omega_{\ell, K-1},  \ 0\leq \rho_K \leq 1$  and
$|\grad_X \rho_K| \leq C$.
\par Testing the  Euler-Lagrange  equation  satisfied by $u_\ell$ with $g_\ell^K$ 
we get

\begin{equation*}
\int_{\Omega_\ell}  A\nabla u_\ell\cdot \nabla g_\ell^K = \lambda_\ell \int_{\Omega_\ell}|\rho_{K}|^2u_\ell^2\,,
\end{equation*}
i.e.,
\begin{equation} \label{ll}
 \int_{\Omega_\ell} A \nabla (\rho_{K}u_\ell) \cdot\nabla (\rho_{K}u_\ell)  -
\int_{\Omega_\ell}u_\ell^2  A\nabla \rho_{K} \cdot\nabla  \rho_{K} = \lambda_\ell \int_{\Omega_\ell} |\rho_{K}|^2u_\ell^2\,.
\end{equation}
Since $g_\ell^K \in V( \Omega_{\ell, K} )$  we get from
\eqref{eq:s_K-Om} that 
\begin{equation}\label{diri}
s_{\ell,K}  \int_{\Omega_{\ell, K}}u_\ell^2 |\rho_{K}|^2 \leq
\int_{\Omega_{\ell, K}} A\nabla(\rho_{K} u_\ell)\cdot \nabla (\rho_{K}  u_\ell) \,.
\end{equation}
Combining \eqref{ll}--\eqref{diri}, we obtain, for some positive constant $C_A$,
\begin{equation*}\label{zero1}
\begin{aligned}
(s_{\ell,K} -\lambda_\ell ) \int_{\Omega_{\ell, K-1}} u_\ell^2 \leq  &(s_{\ell,K} -\lambda_\ell ) \int_{\Omega_{\ell, K}} u_\ell^2 |\rho_{K}|^2 \leq
\int_{\Omega_{\ell,K}}u_\ell^2 (A\nabla \rho_{K})\cdot \nabla \rho_{K}\\ &= \int_{
  \Omega_{\ell, K} \setminus \Omega_{\ell, K-1 }}u_\ell^2 (A\nabla
\rho_{K})\cdot \nabla \rho_{K}
\leq C_A  \int_{  \Omega_{\ell,K } \setminus \Omega_{\ell, K-1} }u_\ell^2\,.
\end{aligned}
\end{equation*}
 From \rlemma{lem:Ktoi}, \rlemma{conver} and \eqref{contr} we get that
 for  $\ell>\ell_0$ there holds
 \begin{equation}\label{next}s_{\ell,K} -\lambda_\ell > \mu^0>0, \ \textrm{ for  some constant } \mu^0.
 \end{equation}
Therefore we
deduce that
\begin{equation*}
(C_A+\mu^0)\int_{\Omega_{\ell, K-1}} u_\ell^2  \leq C_A \int_{\Omega_{\ell, K}}u_\ell^2\,,
\end{equation*}
i.e.,
\begin{equation}
  \label{iteration}
  \int_{\Omega_{\ell, K-1}} u_\ell^2\le \alpha  \int_{\Omega_{\ell, K}} u_\ell^2\,,
\end{equation}
with  $\alpha :=\frac{C_A}{C_A+\mu^0}<1$. 
 Applying \eqref{iteration} successively for $K= \ell^\beta +1, \ell^\beta +2, \ldots,  \ell^\beta + [\ell^\beta]$ yields
\begin{equation*}
\int_{\Omega_{\ell, \ell^\beta}} u_\ell^2 \leq e^{[\ell^\beta]\log\alpha} \int_{\Omega_\ell}u_\ell^2  = e^{[\ell^\beta ]\log\alpha},
\end{equation*}
and \eqref{eq:estimate} follows.
\par Finally we note that
the $m-$dimensional area of $\partial(\ell\omega_1)$ is
$\sim\ell^{m-1}$. Hence 
we may choose $N_\ell$ points $\{P_i\}_{i=1}^{N_\ell}\in\partial\omega_1$ with $N_\ell\le
C\ell^{m-1-\beta}$ such that
\begin{equation}
  \label{eq:17}
  \Omega_\ell\setminus\Omega_{\ell-\ell^\beta}\subset \bigcup_{i=1}^{N_\ell}\Omega_{\ell,\ell^\beta}^{P_i}\,.
\end{equation}
Combining \eqref{eq:estimate} with \eqref{eq:17} yields
\begin{equation*}
  \int_{ \Omega_\ell\setminus\Omega_{\ell-\ell^\beta}}u_\ell^2\le \sum_{i=1}^{N_\ell}\int_{\Omega_{\ell,\ell^\beta}^{P_i}}u_\ell^2\le C\ell^{m-1-\beta}e^{[\ell^\beta ]\log\alpha},
\end{equation*}
and \eqref{po} follows.
\end{proof}
\\[2mm]
Now we are ready to present the proof of \rth{lb}.\\[2mm]
\begin{proof}[Proof of \rth{lb}]
  Assume by contradiction that
  \begin{equation}
    \label{eq:15}
    \liminf_{\ell\to\infty}\lambda_\ell<\min_{\nu\in S^{m-1}}Z^\nu.
  \end{equation}
  The assumption  $A_{12}\nabla_\xi W\not\equiv0$ implies that
  there exists $\nu_0\in S^{m-1}$ for which $\nu_0^T A_{12}\nabla_\xi
  W\not\equiv0$. This implies, by combining \cite[Thm~4.2]{pr} with
   \cite[Thm~5.2]{pr} for the operator associated with the matrix
   $A_{\nu_0}$ (see \eqref{eq:4}) that we have
    \begin{equation*}
      \min\{Z^{\nu_0},Z^{-\nu_0}\}<\mu_1\,.
    \end{equation*}
    Applying \rlemma{l5} we deduce from \eqref{eq:15}  that also
    \begin{equation}
      \label{eq:39}
      \liminf_{\ell\to\infty}\lambda_\ell<\min_{\nu\in S^{m-1}}{\widetilde
        Z}^\nu. 
    \end{equation}
    By \rlemma{nm} we get for $k$ sufficiently large
    \begin{equation}
      \label{eq:27}
     \int_{ \Omega_{\ell_k} \setminus \Omega_{\ell_k -\ell_k^\beta}} u_
{\ell_k}^2 <\frac{1}{4}.
    \end{equation}
    But by \rlemma{ef} we also have,    for $k$ sufficiently large,
    \begin{equation}
      \label{eq:28}
      \int_{\Omega_{\ell_k -\ell_k^\beta}} u_
{\ell_k}^2 <\frac{1}{4}.
\end{equation}
Combining \eqref{eq:27} with \eqref{eq:28} we get for large $k$
\begin{equation*}
  1=\int_{\Omega_{\ell_k}}u_
{\ell_k}^2 <\frac{1}{2},
\end{equation*}
which is clearly a contradiction.
\end{proof}
\\[2mm]

\begin{proof}[Proof of Theorem \ref{main}]
  Assume first that \eqref{co} holds. In this case  it suffices to
  combine \rth{upperbound} with \rth{lb} to get the result. If
  \eqref{co} does not hold, i.e., we have
  \begin{equation}
    \label{eq:34}
     A_{12} \grad_{\xi} W  = 0  \ a.e. \ \text{ in }  \ \omega_2,
   \end{equation}
   Then it can be easily verified that $u(x)=W(\xi)$ (see
   \eqref{eq:54}) is a positive eigenfunction in \eqref{eq:5} with
   $\sigma=\mu_1$, whence $\lambda_\ell=\mu_1$ for all $\ell$.
 \end{proof}

\section{Limit of the  Higher Order Eigenvalues}
\label{sec:high}

This short section is devoted to the proof of \rth{high}. Recall the
Rayleigh quotient characterization of $\lambda_\ell^k$ for any $k\ge2$:
\begin{equation}
  \label{eq:40}
  \lambda_\ell^k=\inf\left\{\frac{\int_{\Omega_\ell}(A\nabla
    u)\nabla u}{\int_{\Omega_\ell}u^2}\,:\,0\ne u\in
  H^1(\Omega_\ell), u=0 \text{ on }\gamma_\ell,  \int_{\Omega_\ell}uu_\ell^i=0 \text{ for }1\le i\le k-1\right\},
\end{equation}
 where $u_\ell^i$ denotes an eigenfunction in \eqref{eq:5} corresponding to $\sigma=\lambda_\ell^i$.
Roughly
speaking, the case $m\ge2$ is easier than the case $m=1$ since we have
\enquote{more space} to carry out a construction of competitors for
the infimum in \eqref{eq:40}.\\[2mm]
\begin{proof}[Proof of \rth{high}]
  Since $\lambda_\ell^k \geq \lambda_\ell^1=\lambda_\ell$ for all $k$, 
  the lower bound $\liminf_{\ell \rightarrow \infty}\lambda_\ell^k \geq \min_{\nu \in S^{m-1}} Z^\nu$ follows from Theorem \ref{main}.
  It remains to prove the other inequality, namely, that for every
  $k\ge2$ there holds
  \begin{equation}
    \label{eq:35}
    \liminf_{\ell \rightarrow \infty}\lambda_\ell^k \le \min_{\nu \in S^{m-1}} Z^\nu\,.
  \end{equation}
Suppose that $k$ distincts $\{\nu_j\}_{j=1}^k\in S^{m-1}$  have been fixed. By
the proof of \rth{upperbound}  we can construct 
functions  $w_\ell^j \in V(\Omega_\ell)$, for $j=1 ,  \ldots  ,  k$, 
with 
\begin{equation}
\label{2nd}
 \int_{\Omega_\ell} |w_\ell^j|^2 = 1 \  \textrm{and}   \ \limsup_{\ell \rightarrow \infty} \int_{\Omega_\ell} A\grad w_\ell^j\cdot \grad w_\ell^j \leq Z^{\nu_j}.
\end{equation}
Basically, for each $j$, as $\ell\to\infty$    the functions  $w_\ell^j$ concentrate  near the point $\ell P_0^j \in \del\Omega_\ell$, where  $P_0^j \in \del\omega_1 \times \omega_2$ is such that the $\nu_j$ is the outward unit  normal at this point.

Next define $W_\ell = \sum_{j=1}^{k} \alpha_\ell^j w_\ell^j \in
V(\Omega_\ell)$  where $\{\alpha_\ell^j\}_{j=1}^k$ are to be chosen  appropriately.
In order to use $W_\ell$ as a competitor in \eqref{eq:40} we  need to
choose $\{\alpha_\ell^j\}_{j=1}^k$
so that the following equalities hold:
\begin{equation}
  \label{eq:44}
  \int_{\Omega_\ell} W_\ell u_\ell^j = 0,\quad  1\le j\le k-1.
\end{equation}
A nontrivial choice for $\{\alpha_\ell^j\}_{j=1}^k$  is possible
because  \eqref{eq:44} induces a linear  system of $k-1$
equations with  $k$ unknowns.

For any given $\varepsilon>0$ we may assume $\ell$ is large enough so
that
\begin{equation}
  \label{eq:41}
  \int_{\Omega_\ell} A\grad w_\ell^j\cdot \grad w_\ell^j\le
  Z^{\nu_j}+\varepsilon,\quad 1\le j\le k.
\end{equation}
Plugging the resulting $W_\ell$ in
\eqref{eq:40} yields
\begin{equation}\label{kh}\lambda_\ell^k \leq  \frac{\sum_{j=1}^{k}(
    \alpha_\ell^j)^2 Z^{ \nu_j} }{\sum_{j=1}^k( \alpha_\ell^j)^2 } +
  \varepsilon \leq \max_{1\le j\le k}\{ Z^{\nu_j} \} +\varepsilon.\end{equation}

Since  the map $\nu \mapsto Z^\nu$ is
continuous on $S^{m-1}$ by \rlemma{lem:Z-nu} we can now conclude using \eqref{kh}. Indeed,
for any $\nu\in S^{m-1}$ we can choose distinct
$\{\nu_j\}_{j=1}^k\subset S^{m-1}$ (all of them \enquote{close to $\nu$}) such that
\begin{equation}
  \label{eq:42}
  |Z^{\nu_j}-Z^{\nu}|<\varepsilon,\quad 1\le j\le k..
\end{equation}
Using this choice of $\{\nu_j\}_{j=1}^k$ in the above
construction  yields, combining \eqref{eq:42} and \eqref{kh}, 
\begin{equation*}
  \lambda_\ell^k\le Z^\nu+2\varepsilon.
\end{equation*}
 Since this can be done for any $\nu\in S^{m-1}$ the desired result
 \eqref{eq:35} follows.
\end{proof}

\end{document}